\documentclass{amsproc}

\usepackage{amssymb}

\usepackage{}

\newtheorem{theorem}{Theorem}[section]
\newtheorem{lemma}[theorem]{Lemma}

\theoremstyle{definition}

\theoremstyle{remark}
\newtheorem{remark}[theorem]{Remark}

\numberwithin{equation}{section}

\newcommand{\R}{\ensuremath{\mathbb{R}}}
\newcommand{\C}{\ensuremath{\mathbb{C}}}
\usepackage{xcolor}  % REMOVE WHEN FINISHED 

\begin{document}

\title[Lagrangian submanifolds of the complex quadric]{Lagrangian submanifolds of the complex quadric \\ as Gauss maps of hypersurfaces of spheres}


%    author one information
\author{Joeri Van der Veken}
\address{KU Leuven, Department of Mathematics, Celestijnenlaan 200B - Box 2400, 3001 Leuven, Belgium}
\curraddr{}
\email{joeri.vanderveken@kuleuven.be}
\thanks{The first author is supported by the Excellence Of Science project G0H4518N of the Belgian government and both authors are supported by project 3E160361 of the KU Leuven Research Fund.}

%    author two information
\author{Anne Wijffels}
\address{KU Leuven, Department of Mathematics, Celestijnenlaan 200B - Box 2400, 3001 Leuven, Belgium}
\curraddr{}
\email{anne.wijffels@kuleuven.be}
\thanks{}

%    The 2010 edition of the Mathematics Subject Classification is
%    the current definitive version.

\date{}

\begin{abstract}
The Gauss map of a hypersurface of a unit sphere $S^{n+1}(1)$ is a Lagrangian immersion into the complex quadric $Q^n$ and, conversely, every Lagrangian submanifold of $Q^n$ is locally the image under the Gauss map of several hypersurfaces of $S^{n+1}(1)$. In this paper, we give explicit constructions for these correspondences and we prove a relation between the principal curvatures of a hypersurface of $S^{n+1}(1)$ and the local angle functions of the corresponding Lagrangian submanifold of $Q^n$. The existence of such a relation is remarkable since the definition of the angle functions depends on the choice of an almost product structure on $Q^n$ and since several hypersurfaces of $S^{n+1}(1)$, with different principal curvatures, correspond to the same Lagrangian submanifold of $Q^n$.
\end{abstract}

\maketitle


%%%%%%%%%%%%%%%%%%%%%%%%%%%%%%%%%%%%%%%%%%%%%
\section{The geometry of the complex quadric}
%%%%%%%%%%%%%%%%%%%%%%%%%%%%%%%%%%%%%%%%%%%%%

Let $\C P^{n+1}(4)$ be the complex projective space of complex dimension $n+1$ equipped with the Fubini-Study metric $g_{FS}$ of constant holomorphic sectional curvature $4$. Then the Hopf fibration $\pi : S^{2n+3}(1) \subseteq \C^{n+2} \to \C P^{n+1}(4): z \mapsto [z]$ is a Riemannian submersion from the unit sphere of real dimension $2n+3$ to $\C P^{n+1}(4)$. Remark that for any $z \in S^{2n+3}(1)$ we have $\pi^{-1}\{[z]\} = \{e^{it}z \ | \ t\in\R\}$ and $\ker (d\pi)_z = \mathrm{span} \{iz\}$. The complex structure $J$ on $\C P^{n+1}(4)$ is induced from multiplication by $i$ on $TS^{2n+3}(1)$ and it is well-known that $(\C P^{n+1}(4), g_{FS}, J)$ is a K\"ahler manifold.

We define the \textit{complex quadric} of complex dimension $n$ as the following complex hypersurface of $\C P^{n+1}(4)$:
\begin{equation*} \label{def_Qn}
Q^n = \{ [(z_0,\ldots,z_{n+1})] \in \C P^{n+1}(4) \ | \ z_0^2 + \ldots + z_{n+1}^2 = 0 \}.
\end{equation*}
If $Q^n$ is equipped with the induced metric $g_{FS}|_{Q^n}$, which we will denote by $g$, and the induced almost complex structure $J|_{Q^n}$, which we will again denote by $J$, then $(Q^n,g,J)$ is of course a K\"ahler manifold itself. The inverse image of $Q^n$ under the Hopf fibration is the $(2n+1)$-dimensional Stiefel manifold
\begin{equation*}
V^{2n+1} \! = \! \left\{ u+iv \, \left| \, u,v\in\R^{n+2}, \, \langle u,u \rangle = \langle v,v \rangle = \frac 12, \, \langle u,v \rangle = 0 \right. \right\} \subseteq S^{2n+3}(1),
\end{equation*}
where $\langle \cdot,\cdot \rangle$ denotes the Euclidean inner product on $\R^{n+2}$. From this perspective, it is easy to see that $Q^n$ can be identified with the Grassmannian of oriented $2$-planes in $\R^{n+2}$ and hence, as a homogeneous space, is
\begin{equation*} 
Q^n = \frac{\mathrm{SO}(n+2)}{\mathrm{SO}(n) \times \mathrm{SO}(2)}. 
\end{equation*}

Denote by $\mathcal A$ the set of all shape operators of $Q^n$ in $\C P^{n+1}(4)$ associated with unit normal vector fields. Since we need it in the next sections, we allow for elements of $\mathcal A$ to be defined only on a subset of $Q^n$. One can deduce the following (see for example \cite{Reckziegel1996} or \cite{Smyth1967}).

\begin{lemma} \label{lem1}
Any $A \in \mathcal A$ is involutive, symmetric and anti-commutes with $J$. 
\end{lemma}

This implies in particular that $\mathcal A$ is a family of almost product structures. However, these almost product structures are not integrable. In fact, we have the following equalities, which can be found in \cite{Smyth1967}.

\begin{lemma} \label{lem2}
Let $\zeta$ be a unit normal vector field along $Q^n$ in $\C P^{n+1}(4)$ with corresponding shape operator $A$. Then there exists a non-zero one-form $s$ such that $\nabla^{\C P^{n+1}(4)}_X \zeta = - AX + s(X) J\zeta$ and $\nabla^{Q^n}_X A = s(X) JA$ for all $X$ tangent to $Q^n$, where $\nabla^{\C P^{n+1}(4)}$ and $\nabla^{Q^n}$ are the Levi Civita connections of $\C P^{n+1}(4)$ and $(Q^n,g)$ respectively.
\end{lemma}

The equation of Gauss for $Q^n$ as a submanifold of $\C P^{n+1}(4)$ yields the following expression for the Riemann-Christoffel curvature tensor of $Q^n$:
\begin{align} \label{R_Qn}
R & ^{Q^n}(X,Y)Z = g(Y,Z)X - g(X,Z)Y \nonumber \\
& + g(JY,Z)JX - g(JX,Z)JY - 2g(JX,Y)JZ \\
& + g(AY,Z)AX - g(AX,Z)AY + g(JAY,Z)JAX - g(JAX,Z)JAY, \nonumber
\end{align}
where $A$ is any element of $\mathcal A$. It follows directly from \eqref{R_Qn} that $Q^n$ is Einstein.

\begin{remark}
Although the almost product structures in $\mathcal A$ are non-integrable, the complex quadric of complex dimension $2$ is in fact a Riemannian product. Indeed, it was proven in \cite{Jensen1969} that a homogeneous Einstein manifold of real dimension $4$ must have either constant sectional curvature or constant holomorphic sectional curvature, or must be a Riemannian product of two surfaces of equal constant Gaussian curvature $S^2(c) \times S^2(c)$ or $H^2(c) \times H^2(c)$. It follows from \eqref{R_Qn} that $Q^2$ does not have constant (holomorphic) sectional curvature and from computing the maximal sectional curvature, we see that $Q^2 = S^2(4) \times S^2(4)$. 
\end{remark}

%%%%%%%%%%%%%%%%%%%%%%%%%%%%%%%%%%%%%%%%%%%%%%%%%%%%%%%%
\section{Lagrangian submanifolds of the complex quadric}
%%%%%%%%%%%%%%%%%%%%%%%%%%%%%%%%%%%%%%%%%%%%%%%%%%%%%%%%

An isometric immersion $f: M^n \to Q^n$ of a manifold of real dimension $n$ into $Q^n$ is said to be \textit{Lagrangian} if $J$ maps the tangent space to $M^n$ at any point into the normal space to $M^n$ at that point and vice versa. If $f: M^n \to Q^n$ is such a Lagrangian submanifold and $A \in \mathcal A$ is defined at least along $f(M^n)$, it was proven in \cite{Li2018} that, in a neighborhood of any point of $M^n$, there exist an orthonormal frame $\{e_1,\ldots,e_n\}$ on $M^n$ and local angle functions $\theta_1,\ldots,\theta_n$ such that 
\begin{equation} \label{eq:anglefunctions}
A(df)e_j = \cos(2\theta_j) (df)e_j - \sin(2\theta_j) J(df)e_j
\end{equation}
for all $j=1,\ldots,n$. Clearly, the angle functions are only defined up to addition with an integer multiple of $\pi$ and they depend on the choice of $A$.

\begin{remark}[Choice of $A$ along a Lagrangian submanifold of $Q^n$] \label{canonical_choice}
Assume that, apart from a Lagrangian immersion $f: M^n \to Q^n$, also a horizontal lift $\hat f:~M^n \to V^{2n+1}$ of $f$ is given. It follows from \cite{Reckziegel1985} that any Lagrangian immersion into $Q^n$ locally allows such a horizontal lift. If $M^n$ is simply connected, the horizontal lift can be defined globally. Since the normal space to $V^{2n+1}$ in $S^{2n+3}(1) \subseteq \C^{n+2}$ at a point $z$ is the complex span of $\bar z$, one can take $\zeta$, defined by $\zeta_{f(p)}=(d\pi)_{\hat f(p)}\left(\overline{\hat{f}(p)}\right)$, as a unit normal vector field to $Q^n$ in $\C P^{n+1}(4)$ along the image of $f$ and the corresponding shape operator is given by $
A X = -(d\pi)_{\hat f(p)}\left(\overline{\hat X}\right)$, 
where $X$ is any vector tangent to $Q^n$ at a point $f(p)$ and $\hat X$ is its horizontal lift to $\hat f(p)$. In the special case that $v$ is tangent to $M^n$ at a point $p$, we have 
\begin{equation} \label{eq:canonical_choice}
A(df)_p(v) = -(d\pi)_{\hat f(p)}\left(d\overline{\hat f}\right)_p(v). 
\end{equation}
This $A$ can be extended to an element of $\mathcal A$, defined in a neighborhood of $f(M^n)$.
\end{remark}

%%%%%%%%%%%%%%%%%%%%%%%%%%%%%%%%%%%%%%%%%%%%%%%%%%%%%%%%%%%%%%%%%%%%%%
\section{Lagrangian submanifolds of the complex quadric as Gauss maps}
%%%%%%%%%%%%%%%%%%%%%%%%%%%%%%%%%%%%%%%%%%%%%%%%%%%%%%%%%%%%%%%%%%%%%%

Several possible definitions for the Gauss map of a hypersurface of a round sphere can be found in the literature. We consider here a definition which was studied in \cite{Palmer1997}. Let $a:M^n \to S^{n+1}(1)$ be a hypersurface of a unit sphere and denote by $b$ a unit normal to the hypersurface, tangent to the sphere. Then the Gauss map of $a$ is the following map from $M^n$ to the complex quadric $Q^n$:
$$ G : M^n \to Q^n : p \mapsto [a(p)+ib(p)]. $$
Looking at $a(p)$ and $b(p)$ as vectors in $\R^{n+2}$, one has $(a(p)+ib(p))/\sqrt 2 \in V^{2n+1}$, such that $[a(p)+ib(p)]=[(a(p)+ib(p))/\sqrt 2]$ is indeed an element of $Q^n$. 

An interesting property that this Gauss map shares with the classical Gauss map of a hypersurface of a Euclidean space is that \textit{parallel hypersurfaces have the same Gauss maps}. Indeed, a parallel hypersurface to a given hypersurface $a$ is obtained by, starting at any point $a(p)$ of the hypersurface, traveling over a distance $t$ along a geodesic of the ambient space with the unit normal $b(p)$ as initial velocity. If the ambient space is $S^{n+1}(1)$, it is easy to see that any parallel hypersurface to $a: M^n \to S^{n+1}(1)$ is given by $a_t: M^n \to S^{n+1}(1) : p \mapsto (\cos t) a(p) + (\sin t) b(p)$ for some $t \in \R$. If $|t|$ is small enough, $a_t$ will, at least locally, be an immersion. A straightforward computation shows that $b_t= -(\sin t) a + (\cos t) b$ is a unit normal to $a_t$ such that $b_0$ equals the original $b$. We conclude that the Gauss map of $a_t$ is given by $G_t = [a_t+ib_t]=[e^{-it}(a+ib)]=[a+ib]=G$.

This Gauss map has attracted quite some attention in recent years, especially in the case of isoparametric hypersurfaces of spheres, i.e., hypersurfaces for which all principal curvatures are constant. We mention for example the works \cite{Ma2009}, \cite {Ma2014} and \cite{Ma2015}. In \cite{Siffert2017}, a study of this Gauss map was proposed as a structural approach to gain a better understanding of the notorious family of isoparametric hypersurfaces of spheres. In \cite{Li2018} a correspondence between the principal curvatures of an isoparametric hypersurface of a sphere and the angle functions of its Gauss map in the sense of \eqref{eq:anglefunctions} was given. Indeed, it turns out that the Gauss map of a hypersurface of a sphere is a Lagrangian immersion into $Q^n$. The following theorem includes this statement and, more importantly, generalizes the result from \cite{Li2018} to arbitrary hypersurfaces of spheres. As mentioned in the abstract, this result is remarkable since both the principal curvatures and the angle functions depend on certain choices.

\begin{theorem} \label{theo1}
Let $a: M^n \to S^{n+1}(1)$ be a hypersurface with unit normal~$b$. Then the Gauss map $G: M^n \to Q^n: p \mapsto [a(p)+ib(p)]$ is a Lagrangian immersion. Moreover, if $A$ is chosen as in Remark \ref{canonical_choice} using the canonical horizontal lift 
\begin{equation} \label{Ghat}
\hat G: M^n \to V^{2n+1}: p \mapsto \frac{1}{\sqrt 2}(a(p)+ib(p)),
\end{equation}
then the relation between the principal curvatures $\lambda_1,\ldots,\lambda_n$ of $a$, with respect to the shape operator associated to $b$, and the angle functions $\theta_1,\ldots,\theta_n$ of $G$ is
\begin{equation} \label{theta_lambda_1} 
\lambda_j = \cot\theta_j
\end{equation}
for $j=1,\ldots,n$.

Conversely, if $f: M^n \to Q^n$ is a Lagrangian immersion, then for every point of $M^n$ there exist an open neighborhood $U$ of that point in $M^n$ and an immersion $a: U \to S^{n+1}(1)$ with Gauss map $f|_U$. This immersion is not unique, nor are its principal curvature functions. However, for any choice of $a$, a local frame of principal directions for $a$ is adapted to $f$ in the sense that \eqref{eq:anglefunctions} holds for any choice of $A$ and the principal curvature functions $\lambda_1,\ldots,\lambda_n$ of $a$ are related to the corresponding local angle functions $\theta_1,\ldots,\theta_n$ by
\begin{equation} \label{theta_lambda_2}
\cot(\theta_j-\theta_k) = \pm \frac{\lambda_j \lambda_k + 1}{\lambda_j - \lambda_k}
\end{equation}
for $j,k=1,\ldots,n$ in points where $\lambda_j \neq \lambda_k$.
\end{theorem}

\begin{proof}
Let $\{e_1,\ldots,e_n\}$ be a local orthonormal frame of principal directions on $M^n$ for the immersion $a: M^n \to S^{2n+1}(1)$, say $Se_j=\lambda_je_j$, where $S$ is the shape operator associated to $b$. It follows directly from \eqref{Ghat} that 
\begin{equation} \label{dGhat}
(d\hat G)e_j = \frac{1}{\sqrt 2}(1-i\lambda_j) e_j.
\end{equation}
These vector fields are orthogonal to $i\hat G$ and hence $\hat G$ is indeed horizontal. Moreover, $G$ is Lagrangian since $(d\hat G)e_j$ and $i(d\hat G)e_k$ are orthogonal for all $j,k=1,\ldots,n$. 

For the choice of $A$ given in Remark \ref{canonical_choice}, we have, using \eqref{eq:canonical_choice} and \eqref{dGhat},
\begin{align*}
A & (dG)e_j = -(d\pi) \left(d\overline{\hat G}\right)e_j = -(d\pi) \left( \overline{(d \hat G)e_j} \right) = -(d\pi) \left( \frac{1}{\sqrt 2}(1+i\lambda_j)e_j \right) \\
& = -(d\pi) \left( \frac{1-\lambda_j^2}{1+\lambda_j^2} \, (d \hat G)e_j \! + \! \frac{2\lambda_j}{1+\lambda_j^2} \, i (d \hat G)e_j \! \right) = \frac{\lambda_j^2-1}{\lambda_j^2+1} \, (dG)e_j \! - \! \frac{2\lambda_j}{\lambda_j^2+1} \, J (dG)e_j.
\end{align*}
Comparing this to \eqref{eq:anglefunctions} implies that the angle functions associated to $A$ are determined by
\begin{equation*} 
\cos(2\theta_j) = \frac{\lambda_j^2-1}{\lambda_j^2+1}, \qquad 
\sin(2\theta_j) = \frac{2\lambda_j}{\lambda_j^2+1}
\end{equation*}
and hence $\lambda_j=\cot\theta_j$. 

Conversely, let $f: M^n \to Q^n$ be a Lagrangian immersion and fix a point $p_0 \in M^n$. If, for some open neighborhood $U$ of $p_0$ in $M^n$, the restriction $f|_U$ is the Gauss map of a hypersurface $a:U \to S^{n+1(1)}$ with unit normal $b$, the first part of the proof implies that $U \to V^{2n+1}: p \mapsto (a(p)+ib(p))/\sqrt 2$ must be a horizontal lift of $f|_U$. Finding all the hypersurfaces of $S^{n+1}(1)$ of which $f$ is locally the Gauss map is hence equivalent to finding the local horizontal lifts of $f$ for which the real part is an immersion. 

It follows from \cite{Reckziegel1985} that for every simply connected open neighborhood $U$ of $p_0$ in $M^n$ there exists a horizontal lift $\hat f_0: U \to V^{2n+1}$ of $f$. Moreover, any other horizontal lift of $f$ on $U$ can be written as $\hat f_t = e^{it} \hat f_0$ for some constant $t \in \mathbb R$. Remark that if we split $\hat f_t$ in a real and imaginary part as $\hat f_t=(a_t+ib_t)/\sqrt 2$, then $a_t= (\cos t) a_0 - (\sin t) b_0$ and $b_t= (\sin t) a_0 + (\cos t) b_0$ for all $t$. If $a_0$ and $a_t$ are immersions, they hence define parallel hypersurfaces of $S^{n+1}(1)$. 

In order to investigate the derivative of $a_t$, we remark that 
\begin{equation} \label{eq:atbt}
a_t = \sqrt 2 \, \mathrm{Re} \, \hat f_t = \frac{1}{\sqrt 2} \left( \hat f_t + \overline{\hat f_t} \right), \quad b_t = \sqrt 2 \, \mathrm{Im} \, \hat f_t = -\frac{i}{\sqrt 2} \left( \hat f_t - \overline{\hat f_t} \right).
\end{equation}
Now choose $A_t$ as in Remark \ref{canonical_choice} using the horizontal lift $\hat f_t$. If $\{e_1^{(t)},\ldots,e_n^{(t)}\}$ is a local orthonormal frame adapted to $f$ in the sense of \eqref{eq:anglefunctions}, then there are local functions $\theta_1^{(t)},\ldots,\theta_n^{(t)}$ such that $A_t(df)e_j^{(t)} = \cos(2\theta_j^{(t)})(df)e_j^{(t)} - \sin(2\theta_j^{(t)})J(df)e_j^{(t)}$ for $j=1,\ldots,n$. Taking the horizontal lift to the image of $\hat f_t$ on both sides of this equality yields
\begin{equation} \label{eq:dfbar}
-\left( d \overline{\hat f_t} \right) \! e_j^{(t)} = \cos(2\theta_j^{(t)}) (d\hat f_t)e_j^{(t)} - i \sin(2\theta_j^{(t)}) (d\hat f_t)e_j^{(t)} = e^{-2i\theta_j^{(t)}}(d\hat f_t)e_j^{(t)}
\end{equation}
for $j=1,\ldots,n$. Substituting $\hat f_t = e^{it} \hat f_0$ into \eqref{eq:dfbar} %gives $-(d\overline{\hat f_0})e_j^{(t)} = e^{-2i(\theta_j^{(t)})-t)}(d\hat f)e_j^{(t)}$.
and then applying $d\pi$ gives $A_0(df)e_j^{(t)} = \cos(2(\theta_j^{(t)}-t))(df)e_j^{(t)} - \sin(2(\theta_j^{(t)}-t)) J(df)e_j^{(t)}$, where $A_0$ is chosen as in Remark \ref{canonical_choice} using the horizontal lift $\hat f_0$. This implies that the frame $\{e_1^{(t)},\ldots,e_n^{(t)}\}$ does not depend on $t$ --we will denote it by $\{e_1,\ldots,e_n\}$ from now on-- and that the corresponding angle functions of $A_t$ and $A_0$ are related by 
\begin{equation} \label{eq:thetat_theta0}
\theta_j^{(t)} = \theta_j^{(0)}+t.
\end{equation}
From \eqref{eq:atbt}, \eqref{eq:dfbar} and \eqref{eq:thetat_theta0} we obtain
\begin{equation} \label{eq:dat} 
(da_t)e_j = \frac{1}{\sqrt 2}\left( 1-e^{-2i(\theta_j^{(0)}+t)} \right)(d\hat f_t)e_j
\end{equation}
for $j=1,\ldots,n$. If we choose $t \in \R$ such that $\theta_j^{(0)}(p_0) + t$ is not an integer multiple of $\pi$ for $j=1,\ldots,n$ and, if necessary, we shrink $U$ to $U_t$ such that none of the functions $\theta_j^{(0)} + t$ attains an integer multiple of $\pi$ on $U_t$, then $a_t|_{U_t}$ is an immersion. There are hence infinitely many choices of $t \in \R$ for which $a_t$ is an immersion in a neighborhood of $p_0$.

Now choose any $t \in \mathbb R$ such that $a_t:U_t \subseteq M^n \to S^{n+1}(1)$ is an immersion. In order to find the principal curvatures of $a_t$, we compute the derivative of the corresponding $b_t$. From \eqref{eq:atbt}, \eqref{eq:dfbar}, \eqref{eq:thetat_theta0} and \eqref{eq:dat} we find
\begin{equation} \label{eq:dbt} 
(db_t)e_j = -\frac{i}{\sqrt 2}\left( 1+e^{-2i(\theta_j^{(0)}+t)} \right)(d\hat f_t)e_j = -\cot(\theta_j^{(0)}+t) (da_t)e_j
\end{equation}
This implies that $\{e_1,\ldots,e_n\}$ is a local frame of principal directions for the hypersurface $a_t$ and that the principal curvatures of $a_t$ defined using the shape operator associated to $b_t$ are given by $\lambda_j^{(t)} = \cot(\theta_j^{(0)}+t)$ for $j=1,\ldots,n$. A first issue is that the principal curvatures are only defined up to sign: if we change the orientation of the unit normal, the signs of the principal curvatures change. A second issue is that the local angle functions $\theta_j^{(0)}$ are only defined through the choice of the almost product structure $A_0$. If one chooses an $A \in \mathcal A$, which is at least defined along $f(U_t)$, then there exists a function $\varphi: U_t \to \mathbb R$ such that $A = \cos\varphi A_0 + \sin\varphi JA_0$ along $f(U_t)$ and it was shown in \cite{Li2018} that the local angle functions associated to $A$ are given by $\theta_j = \theta_j^{(0)} - \varphi/2$ for $j=1,\ldots,n$. This implies that the difference of two local angle functions does not depend on the choice of $A$. Hence, using the formula for the cotangent of a difference, we can state that 
$$ \cot(\theta_j \!-\! \theta_k) = \cot((\theta_j^{(0)} \!\! + \! t) \! - \! (\theta_k^{(0)} \!\! + \! t)) = \frac{(\pm\lambda_j^{(t)})(\pm\lambda_k^{(t)})+1}{(\pm\lambda_j^{(t)})-(\pm\lambda_k^{(t)})} = \pm \frac{\lambda_j^{(t)} \lambda_k^{(t)} + 1}{\lambda_j^{(t)} - \lambda_k^{(t)}} $$
for all $j,k=1,\ldots,n$ in those points where $\lambda_j^{(t)} \neq \lambda_k^{(t)}$. In particular, the right hand side does not depend on $t$. In other words: it does not depend on the chosen horizontal lift of $f$, as long as the real part of this lift is an immersion, or, equivalently, it remains invariant when changing from a hypersurface of a sphere to a parallel hypersurface. This last fact can also be checked directly. 
\end{proof}

\bibliographystyle{amsplain}

\begin{thebibliography}{20}

\bibitem{Jensen1969}
G. R. Jensen, \emph{Homogeneous Einstein spaces of dimension four}, J. Differential Geom. \textbf{3} (1969), 309--349. \MR{0261487}.

\bibitem{Li2018}
H. Li, H. Ma, J. Van der Veken, L. Vrancken and X. Wang, \emph{Minimal Lagrangian submanifolds of the complex hyperquadric}, Sci. China Math., to appear. arXiv:1812.07888.

\bibitem{Ma2009}
H. Ma and Y. Ohnita, \emph{On Lagrangian submanifolds in complex hyperquadrics and isoparametric hypersurfaces in spheres}, Math. Z. \textbf{261} (2009), 749--785. \MR{2480757}.
	
\bibitem{Ma2014}
H. Ma and Y. Ohnita, \emph{Hamiltonian stability of the Gauss images of homogeneous isoparametric hypersurfaces. I}, J. Differential Geom. \textbf{97} (2014), 275--348. \MR{3263508}.
	
\bibitem{Ma2015}
H. Ma and Y. Ohnita, \emph{Hamiltonian stability of the Gauss images of homogeneous isoparametric hypersurfaces II}, Tohoku Math. J. (2) \textbf{67} (2015), 195--246. \MR{3365370}.

\bibitem{Palmer1997}
B. Palmer, \emph{Hamiltonian minimality and Hamiltonian stability of Gauss maps}, Differential Geom. Appl. \textbf{7} (1997), 51--58. \MR{1441918}.

\bibitem{Reckziegel1985} H. Reckziegel, \emph{Horizontal lifts of isometric immersions into the bundle space of a pseudo-Riemannian submersion}, Global Differential Geometry and Global Analysis 1984, Lecture Notes in Mathematics, vol. 1156, Springer, Berlin, Heidelberg, 1985, 264--279. \MR{0824074}.

\bibitem{Reckziegel1996}
H. Reckziegel, \emph{On the geometry of the complex quadric}, Geometry and Topology of Submanifolds, {VIII} (Brussels, 1995 / Nordfjordeid, 1995), World Sci. Publ., River Edge, NJ, 1996, pp.~302--315. \MR{1434581}.
	
\bibitem{Siffert2017}
A. Siffert, \emph{A new structural approach to isoparametric hypersurfaces in spheres}, Ann. Global Anal. Geom. \textbf{52} (2017), 425--456. \MR{3735906}.
	
\bibitem{Smyth1967}
B. Smyth, \emph{Differential geometry of complex hypersurfaces}, Ann. of Math. (2) \textbf{85} (1967), 246--266. \MR{0206881}.
	
\end{thebibliography}

\end{document}